\documentclass{amsart}
\usepackage{setspace}
\usepackage{a4}
\usepackage{amssymb,amsmath,amsthm,latexsym}
\usepackage{amsfonts}
\usepackage{amsfonts}
\usepackage{graphicx}
\usepackage{textcomp}
\usepackage{cite}
\usepackage{enumerate}
\usepackage[mathscr]{euscript}
\usepackage{mathtools}
\newtheorem{theorem}{Theorem}[section]

\newtheorem{corollary}[theorem] {Corollary}
\newtheorem{definition}[theorem]{Definition}
\newtheorem{example}[theorem]{Example}

\newtheorem{question}[theorem] {Question}
\title{This is the title}
\usepackage{amssymb}
\usepackage{amssymb}
\usepackage{amssymb}
\usepackage{amssymb}
\usepackage{amsmath}
\usepackage{tikz}
\usepackage{hyperref}
\usepackage{enumerate}
\usepackage{mathtools}
\usepackage{amsmath}
\usepackage{tikz}
\usepackage{amssymb}
\usepackage{amsmath}
\usepackage{tikz}
\usepackage{hyperref}
\usepackage{amssymb}
\usepackage{amssymb}
\usepackage{amssymb}
\usepackage{amssymb}
\usepackage{amsmath}
\usepackage{tikz}
\usepackage{hyperref}
\usepackage{enumerate}
\usepackage{mathtools}
\usepackage{amsmath}
\usepackage{tikz}
\usepackage{graphicx}
\usepackage{textcomp}
\usetikzlibrary{cd}
\begin{document}
\begin{center}
{\Large{\bf{Expansion of approximate Bessel sequences to approximate Schauder frames for Banach spaces}}}\\
K. Mahesh Krishna and P. Sam Johnson \\
Department of Mathematical and Computational Sciences\\ 
National Institute of Technology Karnataka (NITK), Surathkal\\
Mangaluru 575 025, India  \\
Emails: kmaheshak@gmail.com,  sam@nitk.edu.in\\

Date: \today
\end{center}

\hrule
\vspace{0.5cm}
%--------------------------------------
\textbf{Abstract}: It is known  in Hilbert space frame theory that a Bessel sequence can be expanded to a frame. Contrary to Hilbert space situation, using a result of Casazza and Christensen, we show that there are Banach spaces and approximate Bessel sequences which can not be expanded to approximate Schauder frames. We characterize Banach spaces in which one can expand approximate Bessel sequences to approximate Schauder frames.



%\tableofcontents

\section{Introduction}
 A sequence $\{\tau_n\}_n$ in a separable Hilbert space  $\mathcal{H}$ over $\mathbb{K}$ ($\mathbb{R}$ or $\mathbb{C}$) is said to be a Bessel sequence 
 for $\mathcal{H}$ if there exists $b>0$ such that $\sum_{n=1}^\infty |\langle h, \tau_n\rangle|^2\leq b\|h\|^2, $ $ \forall h \in \mathcal{H} $ and $\{\tau_n\}_n$ is said to be a frame for $\mathcal{H}$ (refer \cite{DUFFINSCHAEFFER, OLE, HEIL} for the basic theory of Bessel sequences and frames)
 if there exist $a,b>0$ such that $a\|h\|^2 \leq \sum_{n=1}^\infty |\langle h, \tau_n\rangle|^2\leq b\|h\|^2, $ $ \forall h \in \mathcal{H}.$ It follows trivially that  every frame is a Bessel sequence. However, a Bessel sequence need not be a frame. Thus it is natural to ask the following question: Given a Bessel sequence, can we add extra elements to it so that the resulting sequence is a frame? Answer is positive. This was first obtained in the theory of operator-valued frames/G-frames \cite{SUN, KAFTAL, SUN2}  by Li and Sun \cite{LISUN} whose particular case says that it is possible to expand every Bessel sequence to a frame. For  finite dimensional Hilbert spaces, an independent proof  was given by Casazza and Leonhard \cite{CASAZZALEONHARD}. Following these works,   expansions of Bessel sequences to dual frames \cite{CHRISTENSENKIMKIM, KOOLIM, BAKICBERIC}, shift-invariant Bessel sequences to frames for $\mathcal{L}^2(\mathbb{R})$ \cite{BOWNIKYU} and UEP-type Bessel sequences (UEP stands for unitary extension principle) to Parseval frames for $\mathcal{L}^2(\mathbb{R})$  \cite{CHRISTENKIMKIMUNITARY} are studied. \\
In this paper, we show negatively that it is not possible to expand approximate Bessel sequences to approximate Schauder frame (ASF) in every Banach space (Corollary \ref{NOT}). We simultaneously  characterize Banach spaces in which approximate Bessel sequences can be expanded to ASFs (Theorem \ref{CHARBESSELTOFRAME}).
\section{Expansions of approximate Bessel sequences for Banach spaces}
 We begin with  the definition of ASF. 
\begin{definition}\cite{FREEMANODELL, THOMAS}\label{ASFDEF}
	Let $\{\tau_n\}_n$ be a sequence in a Banach space  $\mathcal{X}$ and 	$\{f_n\}_n$ be a sequence in  $\mathcal{X}^*$ (dual of $\mathcal{X}$). The pair $ (\{f_n \}_{n}, \{\tau_n \}_{n}) $ is said to be an ASF for $\mathcal{X}$ if the frame operator $	S_{f, \tau}:\mathcal{X}\ni x \mapsto S_{f, \tau}x\coloneqq \sum_{n=1}^\infty
	f_n(x)\tau_n \in
	\mathcal{X}$ is a well-defined bounded linear, invertible operator. Let $\lambda \in \mathbb{K}\setminus \{0\}$. An ASF $ (\{f_n \}_{n}, \{\tau_n \}_{n}) $ is said to be $\lambda$-tight if $S_{f, \tau}=\lambda I_\mathcal{X}$, where $I_\mathcal{X}$ is the identity operator on $\mathcal{X}$.
\end{definition} 
A routine Hilbert space argument shows that a sequence $\{\tau_n\}_n$ is  a Bessel sequence for  Hilbert space  $\mathcal{H}$ if and only if the map $ S_\tau :\mathcal{H} \ni h \mapsto \sum_{n=1}^\infty \langle h, \tau_n\rangle\tau_n\in
\mathcal{H} $ is a well-defined bounded linear operator. In fact, if $\{\tau_n\}_n$ is a Bessel sequence, then both maps $\theta_\tau:\mathcal{H} \ni h \mapsto \theta_\tau h \coloneqq\{\langle h, \tau_n\rangle \}_n \in \ell^2(\mathbb{N})$ and $\theta_\tau^*:\ell^2(\mathbb{N}) \ni \{a_n\}_n \mapsto \theta_\tau^*\{a_n\}_n\coloneqq \sum_{n=1}^{\infty}a_n\tau_n \in \mathcal{H}$  are well-defined bounded linear operators (Chapter 3   in \cite{OLE}). Now $\theta_\tau^*\theta_\tau=S_\tau$ and hence $S_\tau$ is a  well-defined bounded linear operator. Conversely, let $S_\tau$ be a  well-defined bounded linear operator. Definition of $S_\tau$ says that it is a positive operator. Thus there exists $b>0$ such that $\langle S_\tau h, h \rangle \leq b\|h\|^2$, $\forall h \in \mathcal{H}$. Again using the definition of $S_\tau$ gives that  $\{\tau_n\}_n$ is a Bessel sequence. This observation and Definition \ref{ASFDEF} make us to define the following.
 \begin{definition}
 	Let $\{\tau_n\}_n$ be a sequence in a Banach space  $\mathcal{X}$ and 	$\{f_n\}_n$ be a sequence in  $\mathcal{X}^*$. The pair $ (\{f_n \}_{n}, \{\tau_n \}_{n}) $ is said to be an approximate Bessel sequence  for $\mathcal{X}$ if  $ S_{f, \tau}:\mathcal{X}\ni x \mapsto S_{f, \tau}x\coloneqq \sum_{n=1}^\infty
 	f_n(x)\tau_n \in
 	\mathcal{X}$	is a well-defined bounded linear operator.
 \end{definition} 
We next recall the reconstruction property of Banach spaces.
\begin{definition}\cite{CASAZZARECONSTRUCTION}
A Banach space $\mathcal{X}$ is said to have the reconstruction property  if there exists a sequence $\{\tau_n\}_n$   in  $\mathcal{X}$ and a sequence 	$\{f_n\}_n$  in  $\mathcal{X}^*$ such that $x=\sum_{n=1}^\infty
f_n(x)\tau_n ,  \forall x \in \mathcal{X}.$
\end{definition}
Using approximation property of Banach spaces   \cite{CASAZZAAPPROXIMATION} Casazza and Christensen proved the following result.
\begin{theorem}\cite{CASAZZARECONSTRUCTION}\label{RECTHEOREM}
There exists a Banach space $\mathcal{X}$ such that $\mathcal{X}$ does not have the reconstruction property.	
\end{theorem}
 Now we have the following characterization.
\begin{theorem}\label{CHARBESSELTOFRAME}
	Let 	$ (\{f_n \}_{n}, \{\tau_n \}_{n}) $ be an approximate  Bessel sequence for $\mathcal{X}$. Then the following are equivalent.
	\begin{enumerate}[\upshape(i)]
		\item $ (\{f_n \}_{n}, \{\tau_n \}_{n}) $ can be expanded to an ASF for $\mathcal{X}$.
		\item $\mathcal{X}$ has the reconstruction property. 
	\end{enumerate}	 
\end{theorem}
\begin{proof}
\begin{enumerate}[\upshape(i)]
	\item $\Rightarrow $ (ii) Let $\{\omega_n\}_n$ be a sequence in   $\mathcal{X}$ and 	$\{g_n\}_n$ be a sequence in  $\mathcal{X}^*$ such that $ (\{f_n \}_{n}\cup \{g_n \}_{n}, \{\tau_n \}_{n}\cup\{\omega_n \}_{n} ) $ is an ASF   for $\mathcal{X}$. Let $S_{(f,g), (\tau, \omega)}$ be the frame operator for $ (\{f_n \}_{n}\cup \{g_n \}_{n}, \{\tau_n \}_{n}\cup\{\omega_n \}_{n} ) $. Then 
	\begin{align*}
	x&=S_{(f,g), (\tau, \omega)}^{-1}S_{(f,g), (\tau, \omega)}x=S_{(f,g), (\tau, \omega)}^{-1}\left( \sum_{n=1}^{\infty}f_n(x)\tau_n+\sum_{n=1}^{\infty}g_n(x)\omega_n\right)\\
	&=\sum_{n=1}^{\infty}f_n(x)S_{(f,g), (\tau, \omega)}^{-1}\tau_n+\sum_{n=1}^{\infty}g_n(x)S_{(f,g), (\tau, \omega)}^{-1}\omega_n, \quad \forall x \in \mathcal{X} 
	\end{align*}
	which shows that $\mathcal{X}$ has the reconstruction property.
	\item $\Rightarrow $ (i) Let $\{\omega_n\}_n$ be a sequence in   $\mathcal{X}$ and  	$\{g_n\}_n$  be a sequence in  $\mathcal{X}^*$  such that 
$x=\sum_{n=1}^\infty
	g_n(x)\omega_n $,  $ \forall x \in \mathcal{X}.$  Define $h_n\coloneqq g_n $, $\rho_n \coloneqq (I_\mathcal{X}-S_{f, \tau})\omega_n $, for all $n \in \mathbb{N}$. Then 
	
	\begin{align*}
	\sum_{n=1}^{\infty}f_n(x)\tau_n+\sum_{n=1}^{\infty}h_n(x)\rho_n&=\sum_{n=1}^{\infty}f_n(x)\tau_n+\sum_{n=1}^{\infty}g_n(x)(I_\mathcal{X}-S_{f, \tau})\omega_n\\
	&=S_{f, \tau}x+(I_\mathcal{X}-S_{f, \tau})\left(\sum_{n=1}^{\infty}g_n(x)\omega_n\right)\\
	&=S_{f, \tau}x+(I_\mathcal{X}-S_{f, \tau})x=x, \quad \forall x \in \mathcal{X}. 
	\end{align*}
	Therefore $ (\{f_n \}_{n}\cup \{h_n \}_{n}, \{\tau_n \}_{n}\cup\{\rho_n \}_{n} ) $ is an ASF   for $\mathcal{X}$.	
\end{enumerate}	
\end{proof}
 Observe that in the proof of Theorem \ref{CHARBESSELTOFRAME} if we define $h_n\coloneqq g_n (I_\mathcal{X}-S_{f, \tau})$, $\rho_n \coloneqq \omega_n $, for all $n \in \mathbb{N}$, then also $ (\{f_n \}_{n}\cup \{h_n \}_{n}, \{\tau_n \}_{n}\cup\{\rho_n \}_{n} ) $ is an ASF   for $\mathcal{X}$. Also note that, for any nonzero scalar $\lambda$, we can define $\rho_n $ in Theorem \ref{CHARBESSELTOFRAME} as $\rho_n \coloneqq (\lambda I_\mathcal{X}-S_{f, \tau})\omega_n $. Thus, there are  infinitely many ways to expand an approximate  Bessel sequence into an ASF. This result for Hilbert spaces are obtained by Li and Sun \cite{LISUN}.
 \begin{corollary}\label{NOT}
 	There exists a Banach space $\mathcal{X}$ such that given any approximate Bessel sequence $ (\{f_n \}_{n}, \{\tau_n \}_{n}) $  for $\mathcal{X}$, $ (\{f_n \}_{n}, \{\tau_n \}_{n}) $ can not be expanded to an ASF for $\mathcal{X}$.
 \end{corollary}
 \begin{proof}
 From Theorem \ref{RECTHEOREM}, there exists a Banach space $\mathcal{X}$ which does not has the reconstruction property. Let $ (\{f_n \}_{n}, \{\tau_n \}_{n}) $ be any approximate Bessel sequence  for $\mathcal{X}$. Theorem \ref{CHARBESSELTOFRAME} now says that $ (\{f_n \}_{n}, \{\tau_n \}_{n}) $ can not be expanded to an ASF for $\mathcal{X}$.	
 \end{proof}
 Following corollary is an easy consequence of Theorem \ref{CHARBESSELTOFRAME}.
\begin{corollary}
	Let 	$ (\{f_n \}_{n}, \{\tau_n \}_{n}) $ be an approximate  Bessel sequence for $\mathcal{X}$. If 	$\mathcal{X}$ admits a Schauder basis, then $ (\{f_n \}_{n}, \{\tau_n \}_{n}) $ can be expanded to an ASF for $\mathcal{X}$.	
\end{corollary}
Note that  Theorem \ref{CHARBESSELTOFRAME} may not add  countably many  elements to an approximate  Bessel sequence to get an ASF. In the following example we show that  it adds just  one element to an approximate Bessel sequence and yields an ASF. 
 \begin{example}\label{EXAMPLEOME}
 	Let $p\in[1,\infty)$. Let $\{e_n\}_n$ denote the standard  Schauder basis for  $\ell^p(\mathbb{N})$  and $\{\zeta_n\}_n$ denote the coordinate functionals associated with $\{e_n\}_n$. Define 
 	\begin{align*}
 	&R: \ell^p(\mathbb{N}) \ni (x_n)_{n=1}^\infty\mapsto (0,x_1,x_2, \dots)\in \ell^p(\mathbb{N}),\\
 	&L: \ell^p(\mathbb{N}) \ni (x_n)_{n=1}^\infty\mapsto (x_2,x_3,x_4, \dots)\in \ell^p(\mathbb{N}).
 	\end{align*}
 	Clearly  $ (\{f_n\coloneqq \zeta_nL\}_{n}, \{\tau_n\coloneqq Re_n\}_{n}) $ is an approximate Bessel sequence  for 	$\ell^p(\mathbb{N})$. Note that  $S_{f, \tau}=RL$ and 
 	\begin{align*}
 &	(I_{\ell^p(\mathbb{N})}-S_{f, \tau})e_1=e_1-RLe_1=e_1-0=e_1,\\
 &	(I_{\ell^p(\mathbb{N})}-S_{f, \tau})e_n=e_n-RLe_n=e_n-Re_{n-1}=e_n-e_n=0, \quad \forall n \geq 2.
 	\end{align*}
 	Let $g_n\coloneqq \zeta_n$ and $\omega_n\coloneqq e_n$, $\forall n \in \mathbb{N}$. Theorem \ref{CHARBESSELTOFRAME} now says that 
 	 $ (\{f_n \}_{n}\cup \{h_1 \}, \{\tau_n \}_{n}\cup\{\rho_1 \} ) $ is an ASF   for $\ell^p(\mathbb{N})$.
 	\end{example}
 It was proved in \cite{LISUN} that every Gabor frame for $\mathcal{L}^2(\mathbb{R})$ can be expanded to a tight frame for $\mathcal{L}^2(\mathbb{R})$ by adding one window function with the same frequency lattice. We now ask  a similar open problem for  $\mathcal{L}^p(\mathbb{R})$. 
 \begin{question}
 Let $p \in (2, \infty)$, $d \in \mathbb{N}$. For $x\in \mathbb{R}^d$, define $T_x:\mathcal{L}^p(\mathbb{R}^d)\ni f \mapsto T_xf \in \mathcal{L}^p(\mathbb{R}^d)$, $T_xf:\mathbb{R}^d \ni t \mapsto  (T_xf)(t)\coloneqq f(x-t)\in \mathbb{C}$. Let $\{\lambda_n \}_{n}$ be an unbounded sequence in $\mathbb{R}^d$. Choose $f \in \mathcal{L}^p(\mathbb{R}^d)$ and a sequence $\{g_n^*\}_n$ in $ (\mathcal{L}^p(\mathbb{R}^d))^*$  such that  $(\{g_n^*\}_n, \{T_{\lambda_n}f\}_n) $ is an ASF for  $ \mathcal{L}^p(\mathbb{R}^d)$ (such a function and a sequence  exist  \cite{FREEMANODELL}). Can  $(\{g_n^*\}_n, \{T_{\lambda_n}f\}_n) $	be expanded to a tight ASF for $ \mathcal{L}^p(\mathbb{R}^d)$ by adding finitely many elements?
 \end{question}
It was derived by Li and Sun in \cite{LISUN} that if a Bessel sequence for a Hilbert space  can be expanded finitely to get a tight frame, then the number of elements added can not be small. We now derive a similar result for Banach spaces. 
\begin{theorem}\label{NUMBERINEQUALITY}
Let 	$ (\{f_n \}_{n}, \{\tau_n \}_{n}) $ be an approximate  Bessel sequence for $\mathcal{X}$. If  $ (\{f_n \}_{n}\cup \{g_k\}_{k=1}^N, \{\tau_n \}_{n}\cup \{\omega_k\}_{k=1}^N) $ is a $\lambda$-tight ASF for $\mathcal{X}$, then 
\begin{align*}
N\geq \dim (\lambda I_\mathcal{X}-S_{f, \tau}) (\mathcal{X}).
\end{align*}
\end{theorem}
\begin{proof}
Let $S_{(f,g), (\tau, \omega)}$ be the frame operator for $ (\{f_n \}_{n}\cup \{g_k\}_{k=1}^N, \{\tau_n \}_{n}\cup \{\omega_k\}_{k=1}^N) $. Set $S_{g, \omega}(x)\coloneqq\sum_{k=1}^{N}g_k(x)\omega_k,  \forall x \in \mathcal{X}$. Then 
\begin{align*}
\lambda x =S_{(f,g), (\tau, \omega)}x=\sum_{n=1}^{\infty}f_n(x)\tau_n+\sum_{k=1}^{N}g_k(x)\omega_k=S_{f, \tau}x+S_{g,\omega}x,\quad \forall x \in \mathcal{X}.
\end{align*}
Therefore 
\begin{align*}
N \geq \dim S_{g, \omega} (\mathcal{X}) = \dim (\lambda I_\mathcal{X}-S_{f, \tau}) (\mathcal{X}).
\end{align*}	
\end{proof}
Even for Hilbert spaces it is known that the inequality in Theorem \ref{NUMBERINEQUALITY} can not be improved \cite{LISUN} which is illustrated in Example \ref{EXAMPLEOME} as well. Nevertheless, there is a well-behaved class of ASFs, known as p-ASFs \cite{MAHESHJOHNSON} behaving better than ASFs for 
switching between the less known Banach space $\mathcal{X}$ and  the well-known Banach space $\ell^p(\mathbb{N})$.  Moreover, we can characterize p-ASFs and their duals. Note that Hilbert space frame theory is more fertile
due to the fact that we can continuously switch between the Hilbert space and the standard separable Hilbert space $\ell^2(\mathbb{N})$.  We now state the definition of a p-approximate Bessel sequence and end the paper with an open question and a partial answer for it.

\begin{definition}
Let $p \in [1, \infty)$.	An approximate Bessel sequence   $ (\{f_n \}_{n}, \{\tau_n \}_{n}) $  for $\mathcal{X}$	is said to be a p-approximate Bessel sequence for $\mathcal{X}$ if both the maps $
	\theta_f: \mathcal{X}\ni x \mapsto \theta_f x\coloneqq \{f_n(x)\}_n \in \ell^p(\mathbb{N}) $ and $
	\theta_\tau : \ell^p(\mathbb{N}) \ni \{a_n\}_n \mapsto \theta_\tau \{a_n\}_n\coloneqq \sum_{n=1}^\infty a_n\tau_n \in \mathcal{X}$
	are well-defined bounded linear operators. 
\end{definition}
\begin{question}
For which classes of Banach spaces, p-approximate Bessel sequences can be expanded to a p-ASF?	
\end{question}
\begin{theorem}
Let $p \in [1, \infty)$.	If    $ (\{f_n \}_{n}, \{\tau_n \}_{n}) $ is a p-approximate Bessel sequence for $\ell^p(\mathbb{N}), $ then 	$ (\{f_n \}_{n}, \{\tau_n \}_{n}) $ can be expanded to a p-ASF.
\end{theorem} 
\begin{proof}
Let $ \{e_n \}_{n}$ and $ \{\zeta_n \}_{n}$ be as in Example \ref{EXAMPLEOME}.  Define $h_n\coloneqq \zeta_n $, $\rho_n \coloneqq (I_{\ell^p(\mathbb{N})}-S_{f, \tau})e_n $, for all $n \in \mathbb{N}$. Then  it follows that $ (\{f_n \}_{n}\cup \{h_n \}_{n}, \{\tau_n \}_{n}\cup\{\rho_n \}_{n} ) $ is a p-ASF for  $\ell^p(\mathbb{N}) $.
\end{proof}

 \bibliographystyle{plain}
 \bibliography{reference.bib}

\end{document}